\documentclass[review]{amsart}
\usepackage{graphicx}
\newtheorem{thm}{Theorem}[section]

\newtheorem{lem}[thm]{Lemma}
\newtheorem{prop}[thm]{Proposition}
\theoremstyle{definition}

\theoremstyle{remark}
\newtheorem{rem}[thm]{Remark}
\numberwithin{equation}{section}

\newcommand{\Real}{{\bf R}}
\newcommand{\eps}{\varepsilon}

\def\R{{\bf R}}

\def\S{{\bf S}}
\def\d{\displaystyle}
\def\e{{\varepsilon}}

\def\vp{\varphi}
\def\p{\partial}

\begin{document}

\title[]
{On the Blow-up for Critical Semilinear Wave Equations with Damping in the
Scattering Case}
\author{Kyouhei Wakasa and Borislav Yordanov}

\address{Department of Mathematics, Faculty of Science and Technology, 
Tokyo University of Science, 2641 Yamazaki, Noda-shi, Chiba, 278-8510, Japan}
\address{Office of International Affairs, Hokkaido University, 
Kita 15, Nishi 8, Kita-ku, Sapporo, Hokkaido 060-0815, Japan 
and Institute of Mathematics, Sofia}
\vskip10pt
\address{}
\email{wakasa\_kyouhei@ma.noda.tus.ac.jp}
\email{byordanov@oia.hokudai.ac.jp}



\date{\today}
 \keywords{}

\begin{abstract}
We consider the Cauchy problem for semilinear wave equations with variable coefficients and
time-dependent scattering damping in $\Real^n$, where $n\geq 2$.
It is expected that the critical exponent will be Strauss' number $p_0(n)$,  which is
also the one for semilinear wave equations without damping terms.

Lai and Takamura~\cite{LT} have obtained the blow-up part, together with
the upper bound of lifespan, in the sub-critical case $p<p_0(n)$.
In this paper, we extend their results to the critical case $p=p_0(n)$.
The proof is based on \cite{WY}, which concerns the blow-up and upper bound of lifespan
for critical semilinear wave equations with variable coefficients.
\end{abstract}
\maketitle


\section{Introduction}

We study the blow-up problem for critical semilinar wave equations
with variable coefficients and scattering damping depending on time.
The perturbations of Laplacian are uniformly elliptic operators
$$\d \Delta_g =\sum_{i,j=1}^n\partial_{x_i}g_{ij}(x)\partial_{x_j}$$
whose coefficients satisfy, with some $\alpha>0,$ the following:
\begin{equation}
\label{g}
g_{ij}\in C^1(\Real^n),\quad |\nabla g_{ij}(x)|+|g_{ij}(x)-\delta_{ij}|=O(e^{-\alpha|x|})  \hbox{ as } |x|\rightarrow\infty.
\end{equation}
The admissible damping coefficients are $a\in C([0,\infty))$, such that
\begin{equation}
\label{asm1}
\forall t\geq 0\quad a(t)\geq 0 \ \hbox{ and } \ \int_0^\infty a(t) dt<\infty.
\end{equation}

For $n\geq 2$ and $p>1$, we consider the Cauchy problem
\begin{eqnarray}
\label{main}
 u_{tt}-\Delta_g u+a(t)u_t=|u|^{p},& & x\in \Real^n,\quad t>0,\\
\label{data}
u|_{t=0}=\eps u_0,\quad u_t|_{t=0}=\eps u_1,& & x\in \Real^n,
\end{eqnarray}
where $u_0,\: u_1\in C_0^{\infty}(\R^n)$ and $\e>0$ is a small parameter. Our results concern
only the critical case $p=p_0(n)$ with Strauss' exponent defined in (\ref{strauss}) below.

Let us briefly review previous results concerning (\ref{main}) with $g_{ij}=\delta_{ij}$ and various types of
damping $a$.
When $a(t)=1$, Todorova and Yordanov \cite{TY} showed that the solution
of (\ref{main}) blows up in finite time if $1<p<p_F(n)$, where $p_F(n)=1+2/n$ is the Fujita exponent known
to be the critical exponent for the semilinear heat equation.
The same work also obtained global existence for $p>p_F(n)$.
Finally, Zhang \cite{Z} established the blow-up in the critical case $p=p_F(n)$.

The other typical example of effective damping is $a(t)=\mu/(1+t)^{\beta}$ with $\mu>0$ and $\beta\in\R$.
When $-1<\beta<1$, Lin, Nishihara and Zhai \cite{LNZ12} obtained the expected blow-up result, if $1<p\le p_F(n)$,
and global existence result, if $p>p_F(n)$; see also D'Abbicco, S.Lucente and M.Reissig~\cite{DLR13}.

In the case of critical decay $\beta=1$, there are several works about finite time blow-up and global existence.
Wakasugi~\cite{WY14_scale} showed the blow-up, if $1<p\le p_F(n)$ and $\mu>1$ or $1<p\le 1+2/(n+\mu-1)$ and $0<\mu\le 1$.
Moreover, D'Abbicco~\cite{DABI} verified the global existence, if $p>p_F(n)$ and
$\mu$ satisfies one of the following: $\mu\ge5/3$ for $n=1$, $\mu\ge3$ for $n=2$ and
$\mu\ge n+2$ for $n\ge3$. An interesting observation is that the Liouville substitution $w(x,t):=(1+t)^{\mu/2}u(x,t)$
transforms the damped wave equation (\ref{main}) into the Klein-Gordon equation
\[
w_{tt}-\Delta w+\frac{\mu(2-\mu)}{4(1+t)^2}w=\frac{|w|^p}{(1+t)^{\mu(p-1)/2}}.
\]
Thus, one expects that the critical exponent for $\mu=2$ is related to
that of the semilinear wave equation. 
D'Abbicco, Lucente and Reissig \cite{DLR14} have actually obtained
the corresponding blow-up result, if $1<p<p_c(n):=\max\left\{p_{F}(n),\ p_0(n+2)\right\}$ and
\begin{equation}
\label{strauss}
p_0(n):=\frac{n+1+\sqrt{n^2+10n-7}}{2(n-1)}
\end{equation}
is the so-called Strauss exponent, the positive root of the quadratic equation
\begin{equation}
\label{gamma}
\gamma(p,n):=2+(n+1)p-(n-1)p^2=0.
\end{equation}
Their work also showed the existence of global classical solutions for small $\e>0$, if $p>p_c(n)$ and either $n=2$ or $n=3$ and the data are radially symmetric. 
Finally, we mention that our original equations (\ref{main}) is related to semilinear wave equations in 
the Einstein-de Sitter spacetime considered by Galstian $\&$ Yagdjian \cite{GY17}.

We recall that $p_0(n)$ in (\ref{strauss}) is the critical exponent for the semilinear wave equation conjectured by Strauss \cite{St81}. The hypothesis has been verified in several cases; see \cite{WY} and the references therein. 
A related problem is to estimate the lifespan, or the maximal existence time $T_\eps$ of solutions to (\ref{main}), (\ref{data}) in the energy space
$C([0,T_\eps), H^1(\Real^n))\cap C^1([0,T_\eps), L^2(\Real^n)).$

Lai, Takamura and Wakasa \cite{LTW} have obtained the blow-up part of Strauss' conjecture, together with an upper bound of the lifespan $T_\eps$, for (\ref{main}),  (\ref{data})  in the case $n\ge2$, $0<\mu<(n^2+n+2)/2(n+2)$ and $p_F(n)\le p<p_0(n+2\mu)$.
Later, Ikeda and Sobajima \cite{IS} were able to replace these conditions by less restrictive
$0<\mu<(n^2+n+2)/(n+2)$ and $p_F(n)\le p\le p_0(n+\mu)$. In addition, they have derived an
upper bound on the lifespan. Tu and Lin \cite{TL1}, \cite{TL2} have improved the estimates of $T_\eps$ in \cite{IS} recently.

For $\beta\leq -1$, the long time behavior of solutions to (\ref{main}),  (\ref{data}) is quite different. When $\beta=-1$, Wakasugi \cite{W17} has obtained the global existence for exponents
$p_F(n)<p<n/[n-2]_+$, where
\[
[n-2]_{+}:=
\left\{
\begin{array}{lll}
\infty &\mbox{for}& n=1,2,\\
m/(n-2) &\mbox{for}& n\ge3.
\end{array}
\right.
\]
Ikeda and Wakasugi \cite{IW} have proved that the global existence actually holds for any $p>1$ when $\beta<-1$.

For $\beta>1$, we expect the critical exponent to be exactly the Strauss exponent.
In fact, Lai and Takamura \cite{LT} have shown that certain solutions of (\ref{main}), (\ref{data}) blow up in finite time when $1<p<p_0(n)$. 
Moreover, Liu and Wang \cite{LW18} have just obtained the global existence results for $n=3,4$ and $p>p_0(n)$ 
on asymptotically Euclidean manifolds.

If $T_\e$ denotes the lifespan of these solutions, then \cite{LT} have also
given  the upper bound $T_\e \le C\e^{-2p(p-1)/\gamma(p,n)}$ for $n\ge2$ and $1<p<p_0(n)$.
This result is probably sharp, since Takamura~\cite{Ta15} proved the same type of estimate in the sub-critical case of Strauss' conjecture for semilinear wave equations without damping. 
However, both the conjecture and lifespan bound
remained open problems in the critical case $p=p_0(n)$.

The purpose of this paper is to verify the blow-up for $p=p_0(n)$ and to give a proof that extends
to more general damping, including $a(t)\sim(1+t)^{-\beta}$ with $\beta>1$.
We also succeed to derive an exponential type upper bound on the lifespan $T_\eps$,
which is the same as that of the Strauss conjecture in the conservative critical case.
Such results are consistent with our knowledge of the linear problem corresponding to (\ref{main}), (\ref{data});
Wirth \cite{Wir1} has shown that energy space solutions scatter, that is approach solutions to the free wave equations, as $t\rightarrow\infty.$

\begin{thm}
\label{thm:main}
Let $n\ge2$, $p=p_0(n)$ and $a(t)$ satisfy (\ref{asm1}).
Assume that both $u_0\in H^1(\R^n)$ and $u_1\in L^2(\R^n)$ are nonnegative, do
not vanish identically and have supports in the ball $\{x\in\R^n:\ |x|\le R_0 \}$,
where $R_0>1$.

If (\ref{main}) has a solution $(u,u_t)\in C([0,T_\eps), H^1(\Real^n)\times L^2(\Real^n))$, such that
\begin{equation}
\label{support}
\mbox{\rm supp}(u,u_t) \subset \{(x,t)\in\R^n\times[0,T_\eps)\ :\ |x|\le t+R\},
\end{equation}
with $R\geq R_0$, then $T_\eps<\infty$. Moreover,
there exist constants $\e_0=\e_0(u_0,u_1,n,p,R,a)$ and $K=K(u_0,u_1,n,p,R,a)$, such that
\begin{equation}
\label{thm:lifespan}
T_\e\le\exp\left(K\e^{-p(p-1)}\right)\ \hbox{ for }\ 0<\e\le\e_0.
\end{equation}
\end{thm}
\begin{rem}
The lifespan estimates is the same as that of the Strauss conjecture in the critical case
of semilinear  wave equations without damping.
For details, see the introduction in \cite{WY}. 
We also note that Liu and Wang \cite{LW18} have obtained the sharp lower bound of the lifespan, 
$T_\e \ge \exp(c\e^{-2})$ if $n=4$ and $p=p_0(4)=2$. 

\end{rem}

Our proof is based on the approach of Wakasa and Yordanov \cite{WY}. Averaging the solution with respect to
a suitable test function, we derive a second-order dissipative ODE which corresponds to equation (\ref{main}).
The key point is to establish lower bounds for the fundamental system of solutions to this ODE; see Lemma \ref{y1y2}.
As a consequence, we can follow \cite{WY} and obtain the same nonlinear integral inequality.
The final blow-up argument also repeats the iteration argument of \cite{WY}.

\section{Test Functions}

Similarly to the proof of \cite{WY}, we first consider the following elliptic problem:
\begin{equation}
\label{eigen-pr}
\Delta_g \varphi_\lambda = \lambda^2\varphi_{\lambda},\quad x\in \Real^n,
\end{equation}
where $\lambda\in (0,\alpha/2].$ As $\lambda|x|\rightarrow\infty$, these $\vp_\lambda(x)$ are asymptotically given by $\varphi(\lambda x),$ with $\vp$ being the standard radial solution to the unperturbed equation $\Delta \vp =\vp$:
\begin{equation}
\label{eigen}
\varphi(x)=\int_{\S^{n-1}}e^{x\cdot\omega}dS_{\omega}\sim c_n |x|^{-(n-1)/2}e^{|x|},
\quad |x|\rightarrow\infty.
\end{equation}

We recall the following result about the existence and main properties of $\varphi_\lambda$.

\begin{lem}
\label{lem1}
Let $n\geq 2$. There exists a solution $\varphi_\lambda\in C^\infty(\Real^n)$ to
(\ref{eigen-pr}), such that
\begin{equation}
\label{lem1:ineq}
|\varphi_\lambda(x)-\varphi(\lambda x)|\leq C_\alpha\lambda^{\theta},\quad x\in \Real^n,
\quad \lambda\in (0,\alpha/2],
\end{equation}
where $\theta\in (0,1]$ and
$\varphi(x)=\int_{\S^{n-1}}e^{x\cdot\omega} dS_\omega\sim c_n|x|^{-(n-1)/2}e^{|x|},$
$c_n>0,$ as $|x|\rightarrow\infty.$

Moreover, $\vp_\lambda(\: \cdot\:)-\vp(\lambda \: \cdot)$ is a continuous $L^\infty (\Real^n)$ valued
function of $\lambda\in (0,\alpha/2]$ and there exist positive constants $D_0,$ $D_1$ and $\lambda_0$, such that
\begin{equation}
\label{2sided}
D_0 \langle \lambda|x|\rangle^{-(n-1)/2}e^{\lambda|x|}\leq \varphi_\lambda(x)\leq D_1
\langle \lambda|x|\rangle^{-(n-1)/2}e^{\lambda|x|}, \quad x\in\Real^n,
\end{equation}
holds whenever $0<\lambda\leq \lambda_0$.
\end{lem}
\begin{proof}
See Lemma 2.2 in \cite{WY}.
\end{proof}

Given $\lambda_0\in (0,\alpha/2]$ and $q>0$, we also introduce the auxiliary functions
\begin{eqnarray}
\label{aq}
\xi_q (x,t) & = & \int_{0}^{\lambda_0}e^{-\lambda(t+R)}\cosh\lambda t \: \varphi_\lambda(x)\lambda^{q}
d\lambda,\\
\eta _q (x,t,s) & = & \int_{0}^{\lambda_0}e^{-\lambda(t+R)}\frac{\sinh \lambda(t-s)}{\lambda(t-s)}
\: \varphi_\lambda(x)\lambda^{q} d\lambda,
\label{bq}
\end{eqnarray}
for $(x,t)\in \R^n\times \R$ and $s\in\R.$ Useful estimates are collected in the next lemma.

\begin{lem}
\label{lem2}
Let $n\geq 2$. There exists $\lambda_0\in (0,\alpha/2]$, such that the following hold:

(i) if $0<q$, $|x|\leq R$ and $0\leq t$, then
\begin{eqnarray*}
\xi_q (x,t) & \geq & A_0,\\
\eta_q (x,t,0) & \geq & B_0\langle t\rangle^{-1};
\end{eqnarray*}

(ii) if $0<q$, $|x|\leq s+R$ and $0\leq s<t$, then
\begin{eqnarray*}
\eta_q (x,t,s) & \geq & B_1 \langle t\rangle^{-1}\langle s\rangle^{-q};
\end{eqnarray*}

(iii) if $(n-3)/2<q$, $|x|\leq t+R$ and $0<t$, then
\begin{eqnarray*}
\eta_q (x,t,t) & \leq & B_2 \langle t\rangle^{-(n-1)/2}\langle t-|x| \rangle^{(n-3)/2-q}.
\end{eqnarray*}
Here $A_0$ and $B_k$, $k=0,1,2,$ are positive constants depending only on $\alpha$, $q$ and $R$,
while $\langle s\rangle =3+|s|$.
\end{lem}

\begin{proof}
See Lemma 3.1 in \cite{WY}.
\end{proof}

The following lemma plays a key role in the proof of Theorem\ref{thm:main}.

\begin{lem}
\label{y1y2}
Let $\lambda>0$ and introduce the ordinary differential operators
\[
L_a=\partial_t^2+a(t)\partial_t-\lambda^2,\quad  L_a^\ast =\partial_s^2-\partial_s a(s)-\lambda^2.
\]
The fundamental system of solutions $\{y_1(t,s;\lambda),y_2(t,s;\lambda)\}$, defined through
\begin{eqnarray*}
L_a y_1(t,s;\lambda)= 0, && y_1(s,s;\lambda)=1,\quad \partial_t y_1(s,s;\lambda)=0,\\
L_a y_2(t,s;\lambda)= 0, && y_2(s,s;\lambda)=0,\quad \partial_t y_2(s,s;\lambda)=1,
\end{eqnarray*}
depends continuously on $\lambda$ and satisfies the following estimates, for $t\geq s\geq 0$:
\begin{eqnarray*}
(i) && y_1(t,s;\lambda) \geq e^{-\|a\|_{L^1}}\cosh\lambda(t-s),\\
(ii) && y_2(t,s;\lambda) \geq e^{-2\|a\|_{L^1}}\frac{\sinh \lambda(t-s)}{\lambda}.
\end{eqnarray*}

Moreover, the conjugate equations and initial conditions hold:
\begin{eqnarray*}
(iii) && L_a^\ast y_2(t,s;\lambda)=0,\\
(iv) &&  y_1(t,0;\lambda)=a(0)y_2(t,0;\lambda)-\p_s y_2(t,0;\lambda).
\end{eqnarray*}
\end{lem}
\begin{proof} See Section~4.
\end{proof}

\section{Proof of Theorem~1.1}

Let $u$ be a weak solution to problem (\ref{main}), defined below, and $\eta_{q}(x,t,t)$ be a test function, defined in Section~2, with $q>-1$. We will show that
\begin{equation}
\label{def:F}
F(t)=\int_{\R^n} u(x,t) \eta_{q}(x,t,t)dx
\end{equation}
satisfies a nonlinear integral inequality which implies finite time blow-up. Our definition of weak solutions
is standard: $(u,u_t)\in C([0,T_\eps), H^1(\Real^n)\times L^2(\Real^n))$ and $\forall \phi\in C_0^\infty(\Real^n\times [0,T_\eps))$ and $t\in (0,T_\eps)$
\begin{eqnarray*}
\nonumber
& & \int u_s(x,t)\phi(x,t) dx-\int u_s(x,0)\phi(x,0) dx\\
\nonumber
& & -\int_0^t \int (u_s(x,s)\phi_{s}(x,s)- g(x)\nabla u(x,s)\cdot\nabla \phi(x,s)-a(s)u_s(x,s)\phi(x,s)) dxds \\
\nonumber
& & =\int_0^\infty \int |u(x,s)|^p \phi(x,s) dxds.
\end{eqnarray*}
In the next result, however, it will be more convenient to work with
\begin{eqnarray}
\nonumber
& & \int (u_s(x,t)\phi(x,t)-u(x,t)\phi_s(x,t)+a(t)u(x,t)\phi(x,t))dx\\
\label{44}
& & -\int (u_s(x,0)\phi(x,0)-u(x,0)\phi_s(x,0)+a(0)u(x,0)\phi(x,0))dx\\
\nonumber
& & +\int_0^t \int u(x,s)(\phi_{ss}(x,s)- \Delta_g\phi(x,s)-(a(s)\phi(x,s))_s)dxds \\
\nonumber
& & =\int_0^\infty \int |u(x,s)|^p \phi(x,s) dxds,
\end{eqnarray}
which follows from integration by parts. We can also use $\phi\in C^\infty(\Real^n\times [0,T_\eps))$, since
$u(\cdot,s)$ is compactly supported for every $s$.

\begin{prop}
\label{prop:identity}
Let the assumptions in Theorem \ref{thm:main} be fulfilled and $q>-1.$
\begin{equation}
\label{final-equal}
\begin{array}{lll}
\d \int_{\R^n}u(x,t) \eta_{q}(x,t,t) dx\\
\qquad \ge
\d \e e^{-\|a\|_{L^1}}\!\!\!\int_{\R^n}\!\!u_0(x)\xi_{q}(x,t)\: dx
+ \e e^{-2\|a\|_{L^1}}t\int_{\R^n}\!\!u_1(x) \eta_{q}(x,t,0) dx\\
\d \qquad \qquad+e^{-2\|a\|_{L^1}}\!\!\!\int_0^t(t-s) \int_{\R^n}|u(x,s)|^p \eta_{q}(x,t,s) dxds
\end{array}
\end{equation}
for all $t\in (0,T_\eps).$
\end{prop}
\begin{proof} We will apply (\ref{44}) to $\phi(x,s)=\vp_\lambda(x)y_2(t,s;\lambda)$, which satisfies
\begin{eqnarray*}
\phi_{ss}(x,s)-\Delta_g\phi(x,s)-(a(s)\phi(x,s))_s &= & 0,\\
a(0)\phi(x,0)-\phi_{s}(x,0)&= & \vp_\lambda(x)y_1(t,0;\lambda),
\end{eqnarray*}
from Lemma~\ref{y1y2} $(iii)$ and $(iv)$, respectively. Then we obtain
\begin{eqnarray*}
\d\int u(x,t)\varphi_{\lambda}(x) dx & = &\e y_1(t,0;\lambda) \int u_0(x)\varphi_{\lambda}(x) dx\\
 & & +\d\e y_2(t,0;\lambda) \int u_1(x)\varphi_{\lambda}(x)dx\\
& & +\int_{0}^{t}y_2(t,s;\lambda)\left(\int |u(x,s)|^p\varphi_{\lambda}(x) dx\right)ds,
\end{eqnarray*}
where the initial conditions are determined by (\ref{data}) and the pair $\{y_1,y_2\}$
is defined in Lemma~\ref{y1y2}. Making use of estimates $(i)$ and $(ii)$ in this lemma, we have that
\begin{eqnarray*}
\d\int u(x,t)\varphi_{\lambda}(x) dx & \geq  &\e e^{-\|a\|_{L^1}}\cosh(\lambda t)\int u_0(x)\varphi_{\lambda}(x) dx\\
 & & +\d\e e^{-2\|a\|_{L^1}}\frac{\sinh \lambda t}{\lambda}\int u_1(x)\varphi_{\lambda}(x)dx\\
& & +e^{-2\|a\|_{L^1}}\int_0^t\frac{\sinh \lambda(t-s)}{\lambda}\left(\int |u(x,s)|^p\varphi_{\lambda}(x) dx\right)ds.
\end{eqnarray*}
The lower bound (\ref{final-equal}) follows from multiplying the above inequality by $\lambda^q e^{-\lambda(t+R)}$, integrating on $[0,\lambda_0]$ and interchanging the order of integration between $\lambda$ and  $x$.
Recalling definitions (\ref{aq}) and (\ref{bq}) for $\xi_q$ and $\eta_q$, we complete the proof.
\end{proof}

Similarly to the proof of Proposition 4.2. in \cite{WY}, we obtain the convenient iteration frame by using
Lemma \ref{lem2}.
\begin{prop}
\label{prop:frame}
Suppose that the assumptions in Theorem \ref{thm:main} are fulfilled and choose
$q=(n-1)/2-1/p.$ If $F(t)$ is defined in (\ref{def:F}), there exists a positive constant $C=C(n,p,q,R,a)$, such that
\begin{equation}
\label{frame}
F(t)  \geq \frac{C}{\langle t\rangle}
\int_0^t \frac{t-s}{\langle s\rangle}\frac{F(s)^p}{
(\log \langle s\rangle)^{p-1}}\: ds
\end{equation}
for all $t\in (0,T_\eps).$
\end{prop}

The finite time blow-up and lifespan estimate (\ref{thm:lifespan}) can now be derived following Sections 4 and 5
in \cite{WY}.

\section{Proof of Lemma~2.3}

Let us recall that $\lambda>0$ and  $L_a=(d/dt)^2+a(t)d/dt-\lambda^2.$
There exists a pair of $C^2$-solutions $\{y_1(t,s;\lambda),y_2(t,s;\lambda)\}$ which depends continuously on $\lambda$ and satisfies
\begin{eqnarray*}
L_a y_1= 0, && y_1(s,s;\lambda)=1,\quad y_1'(s,s;\lambda)=0,\\
L_a y_2= 0, && y_2(s,s;\lambda)=0,\quad y_2'(s,s;\lambda)=1,
\end{eqnarray*}
for  $t\geq s\geq 0.$
We will show that $\{y_1(t,s;\lambda),y_2(t,s;\lambda)\}$ behaves similarly to the fundamental system of $L_0$, that is
$\{\cosh \lambda (t-s),\lambda^{-1}\sinh \lambda (t-s)\}$, as $t-s\rightarrow\infty$ and $\lambda\rightarrow 0.$
Our proof gives two-sided bounds and relies only on three identities:
\begin{eqnarray}
\label{id1}
&&(y_1'e^{A(t)})'=\lambda^2 y_1e^{A(t)},\ \hbox{ where } \  A(t)=\int_0^t a(r)dr,\\
\label{id2}
&&\left(y_1 e^{A(t)}-\int_s^t a(r)y_1e^{A(r)}dr \right)''=\lambda^2 y_1 e^{A(t)},\\
\label{id3}
&&y_2'y_1-y_2y_1'=e^{A(s)-A(t)} \ \hbox{ or } \ \left(\frac{y_2}{y_1}\right)'=\frac{e^{A(s)-A(t)}}{y_1^2}.
\end{eqnarray}

To verify claim $(i)$, we observe that $y_1(t_0,s;\lambda)=0$ at some $t_0>s$ leads to a contradiction:
if $t_0$ is the first such number, then $y_1(t,s;\lambda)\geq 0$ for $t\in [s,t_0]$ and (\ref{id1}) imply that
\[
y_1'(t,s;\lambda)e^{A(t)}=\lambda^2\int_s^t y_1(r,s;\lambda)e^{A(r)}dr \geq 0,  \hbox{ so }
y_1'(t,s;\lambda)\geq 0  \hbox{ for }  t\in [s,t_0].
\]
Hence, $y_1(t,s;\lambda)$ is increasing on $[s,t_0]$ and $0=y_1(t_0,s;\lambda)\geq y_1(s,s;\lambda)=1$ can not hold.
The positivity of $y_1(t,s;\lambda)$ also yields, through (\ref{id1}), the positivity of its derivative:
$y_1'(t,s;\lambda)\geq 0$ for all $t\geq s.$

We can now derive an upper bound on $y_1$ using that $y_1''=\lambda^2 y_1-ay_1'\leq \lambda^2y_1$ and
$y_1(s,s;\lambda)=1$, $y_1'(s,s;\lambda)=0$:
\begin{equation}
\label{y1<}
y_1(t,s;\lambda)\leq \cosh\lambda(t-s),\quad t\geq s.
\end{equation}

The lower bound on $y_1$ is a consequence of (\ref{id2}) and the positivity of $y_1(t,s;\lambda)$:
\[
\left(y_1 e^{A(t)}-\int_s^t a(r)y_1e^{A(r)}dr \right)'' \geq \lambda^2\left(y_1 e^{A(t)}-\int_s^t a(r)y_1e^{A(r)}dr \right).
\]
Combining this inequality with the initial values at $t=s$,
\begin{eqnarray*}
\left.\left(y_1 e^{A(t)}-\int_s^t a(r)y_1e^{A(r)}dr \right)\right|_{t=s} &= & e^{A(s)},\\
\left.\frac{d}{dt}\left(y_1 e^{A(t)}-\int_s^t a(r)y_1e^{A(r)}dr \right)\right|_{t=s} &= & 0,
\end{eqnarray*}
we obtain that
\[
y_1(t,s;\lambda)e^{A(t)}-\int_s^t a(\tau)y_1e^{A(\tau)}d\tau \geq e^{A(s)}\cosh\lambda(t-s).
\]
After simplifying,
\begin{equation}
\label{y1>}
y_1(t,s;\lambda) \geq e^{A(s)-A(t)}\cosh\lambda(t-s),\quad t\geq s.
\end{equation}
Since $A(s)-A(t)\geq -\|a\|_{L^1}$ and (\ref{y1<}) holds, claim $(i)$ follows:
\begin{equation}
\label{>y1>}
\cosh\lambda(t-s)\geq y_1(t,s;\lambda) \geq e^{-\|a\|_{L^1}}\cosh\lambda(t-s).
\end{equation}

To check claim $(ii)$, we combine (\ref{y1<}), (\ref{y1>}) and identity (\ref{id3}):
\[
\frac{e^{A(t)-A(s)}}{\cosh^2\lambda(t-s)} \geq \left(\frac{y_2}{y_1}\right)'\geq \frac{e^{A(s)-A(t)}}{\cosh^2\lambda(t-s)}.
\]
Using $A(s)-A(t)\geq -\|a\|_{L^1}$ and integration on $[s,t]$, we have that
\[
\frac{e^{\|a\|_{L^1}}}{\lambda}\tanh \lambda(t-s) \geq \frac{y_2(t,s;\lambda)}{y_1(t,s;\lambda)} \geq
\frac{e^{-\|a\|_{L^1}}}{\lambda} \tanh\lambda(t-s).
\]
The final result follows from (\ref{>y1>}):
\[
e^{\|a\|_{L^1}}\frac{\sinh \lambda(t-s)}{\lambda} \geq y_2(t,s;\lambda) \geq
e^{-2\|a\|_{L^1}}\frac{\sinh \lambda(t-s)}{\lambda}.
\]

Finally, we will show the equalities in $(iii)$ and $(iv).$
Set $y_1(t):=y_1(t,0;\lambda)$ and $y_2(t):=y_2(t,0;\lambda)$. It easy to see that
\[
y_2(t,s,\lambda)
=\d\frac{y_1(t)y_2(s)-y_1(s)y_2(t)}{y'_1(s)y_2(s)-y_1(s)y'_2(s)}
=\d (y_1(s)y_2(t)-y_1(t)y_2(s))e^{A(s)}.
\]
Thus, we can calculate $\partial_s^i y_2(t,s,\lambda)$ with $i=1,2$ as follows:
\begin{eqnarray}
\nonumber
\ \ \p_sy_2(t,s;\lambda) &= & (y'_1(s)y_2(t)-y_1(t)y'_2(s))e^{A(s)}+a(s)y_2(t,s;\lambda),\\
\label{derivative-y_2}\\
\ \ \p_s^2y_2(t,s;\lambda) &=&(y''_1(s)y_2(t)-y_1(t)y''_2(s))e^{A(s)}\nonumber\\
                       & &+a(s)(y'_1(s)y_2(t)-y_1(t)y'_2(s))e^{A(s)}+\p_s(a(s)y_2(t,s;\lambda)).\nonumber
\end{eqnarray}
Noticing that $y_i(s)$ with $i=1,2$ satisfy the differential equation $L_ay_i(s)=0$, we get $(iii)$.

To derive $(iv)$, we just set $s=0$ in (\ref{derivative-y_2}) for $\p_sy_2(t,s;\lambda)$ and use the
initial conditions for $y_i(s)$. The proof is complete.

\end{document}